\documentclass[a4paper,10pt,notitlepage,reqno]{article}
\usepackage[english]{babel}
\usepackage[latin1]{inputenc}
\usepackage[T1]{fontenc}
\usepackage{amssymb,amsthm,amsmath} 
\usepackage{mathtools}
\usepackage{mathrsfs}
\usepackage{enumerate}
\usepackage{authblk}
\usepackage{color}
\usepackage{geometry}
\usepackage{hyperref}
\newcommand{\RR}{{\mathbb{R}}}
\newcommand{\CC}{{\mathbb{C}}}

\newcommand{\cL}{{\mathcal{L}}}
\newcommand{\cF}{{\mathcal{F}}}
\newcommand{\cH}{{\mathcal{H}}}

\newcommand{\Dom}{{\operatorname{Dom}}}

\newcommand{\cf}{\emph{cf.}}

\newcommand{\eps}{\varepsilon}

\newcommand{\ov}{\overline}

\newcommand\la{\lambda}

\renewcommand{\d}{{\textrm{d}}}

\renewcommand\eps{\varepsilon}

\newcommand\sd{\sigma_{\rm disc}}
\newcommand\sess{\sigma_{\rm ess}}

\newcommand\mydot{\,\cdot\,}
\newcommand\ds{\displaystyle}

\newcommand\wh{\widehat}

\theoremstyle{plain} 
\newtheorem{theorem}{Theorem}
\newtheorem{lemma}{Lemma}[section]

\theoremstyle{definition}

\theoremstyle{remark}

\newcommand{\normeq}[1]{
	{\left\vert\kern-0.25ex\left\vert\kern-0.25ex\left\vert #1 
		\right\vert\kern-0.25ex\right\vert\kern-0.25ex\right\vert}}

{\left\lbrace\begin{array}{r@{\hspace{1mm}}ll}}%
	{\end{array}\right.}

\begin{document}
%
\title{\textbf{\Large
Discrete eigenvalues of the spin-boson Hamiltonian 
with two photons: on a problem of Minlos and Spohn
}}

\author{Orif Ibrogimov}
\date{\small 
\begin{quote}
\emph{Department of Mathematics, Faculty 
	of Nuclear Sciences and Physical 
	Engineering, Czech Technical University,
	Trojanova 13, 12000 Prague 2, Czech 
	Republic; ibrogori@fjfi.cvut.cz.
}
\end{quote}
25 February 2019}
\maketitle
\begin{abstract}
\noindent
Under minimal regularity conditions on the photon dispersion and 
the coupling function, we prove that the spin-boson model with 
two massless photons in~$\RR^d$ cannot have more than two bound 
state energies whenever the coupling strength is 
sufficiently strong. 
\end{abstract}
%
\footnotetext{\emph{Keywords}. Spin-boson Hamiltonian, 
	Fock space, spectrum, 
	operator matrix, Schur complement.}
\footnotetext{\emph{2010 Mathematics Subject Classification}. 
	81Q10, 47A10, 70F07, 47G10.}
%
\section{Introduction}\label{sec:intro}
\noindent 
In this paper we are concerned with the discrete spectrum analysis
for the Hamiltonian of a quantum mechanical model which 
describes the interaction between a two-level atom and two massless
photons. The energy operator is obtained from the spin-boson 
Hamiltonian by the compression onto the subspace of two bosons and 
acts on the Hilbert space  which is given by the tensor product of 
$\CC^2$ and the truncated Fock space
\begin{equation}
	\cF_s^2:=\CC \oplus L^2(\RR^d) \oplus L^2_s(\RR^d\times\RR^d).
\end{equation}
Here $L^2_s(\RR^d\times\RR^d)$ stands for the subspace of the 
Hilbert space $L^2(\RR^d\times\RR^d)$ consisting of symmetric 
functions and equipped with the inner product 
\begin{equation}\label{inner.prod.in.Ls2}
	(\phi,\psi)=\frac{1}{2}\int_{\RR^d}\int_{\RR^d}\phi(k_1,k_2)\ov{\psi(k_1,k_2)} \,\d k_1 \d k_2, \quad \phi,\psi\in L_s^2(\RR^d\times\RR^d).
\end{equation}
For $f=\bigl(f^{(\sigma)}_0, f^{(\sigma)}_1, f^{(\sigma)}_2\bigr)\in \CC^2\otimes\cF_s^2$, where $\sigma=\pm$ is the discrete variable, the Hamiltonian of our system is given by the formal expression 
\begin{equation}\label{Hamiltonian}
	\begin{aligned}
		(H_{\alpha}f)^{(\sigma)}_0&=\sigma\eps f^{(\sigma)}_0
		+\alpha\int_{\RR^d}\la(q)f^{(-\sigma)}_1(q)\,\d q,\\
		(H_{\alpha}f)^{(\sigma)}_1(k)&
		=(\sigma\eps+\omega(k))f^{(\sigma)}_1(k)
		+\alpha\la(k)f^{(-\sigma)}_0
		+\alpha\int_{\RR^d}f^{(-\sigma)}_2(k,q)\la(q)\,\d q,\\
		(H_{\alpha}f)^{(\sigma)}_2(k_1,k_2) &
		= (\sigma\eps+\omega(k_1)+\omega(k_2))f^{(\sigma)}_2(k_1,k_2)
		+\alpha\la(k_1)f^{(-\sigma)}_1(k_2)
		+\alpha\la(k_2)f^{(-\sigma)}_1(k_1). 
	\end{aligned}
\end{equation}
Here $\eps$ ($\eps>0$) and $-\eps$ are the excited and the ground state energies of the atom, respectively, $\omega(k)=|k|$ is the photon dispersion relation, $\alpha>0$ is the coupling constant and $\la$ is the coupling function which is given by the product of $\sqrt{\omega(k)}$ with a cut-off function for large $k$. 

In general, the dispersion relation $\omega\geq0$ and the coupling 
function~$\la$ are fixed by the physics of the problem. Motivated 
by different applications of the spin-boson Hamiltonian
one considers them as free parameter functions and imposes only 
some general conditions such as 
\begin{equation}\label{cond:la.la.sqrt.in.L2}
	\la\in L^2(\RR^d), \quad \frac{\la}{\sqrt{\omega}}\in L^2(\RR^d),
\end{equation}
provided that 
\begin{equation}\label{zero.mass}
	\inf_{k\in\RR^d}\omega(k)=0.
\end{equation}
Most of the work on the spectrum of the spin-boson model in the up-to-date 
literature assumes at least~\eqref{cond:la.la.sqrt.in.L2}, or its various 
strengthened versions, where, for example, the second condition in 
\eqref{cond:la.la.sqrt.in.L2}  is replaced by the requirement that
\begin{equation}\label{cond2:la.omega.in.L2}
	\frac{\la}{\omega}\in L^2(\RR^d),
\end{equation}
which is known as the \emph{infrared regularity condition} 
(\cf~\cite{Hirokawa-2001RMP}).

We notice that the natural domain of the unperturbed operator~$H_0$ 
is given by the tensor product of~$\CC^2$ 
with~$\CC\oplus\cH_1\oplus\cH_2$, where~$\cH_1$ and~$\cH_2$ are 
the weighted $L^2$-Hilbert spaces
\begin{equation}\label{Dom:H1}
	\cH_1:=\Big\{f\in L^2(\RR^d): 
	\int|\omega(k)|^2|f(k)|^2\,\d k<\infty\Big\}
\end{equation}
and
\begin{equation}\label{Dom:H2}
	\cH_2:=\Big\{g\in L^2_s(\RR^d\times\RR^d): 
	\int|\omega(k_1)+\omega(k_2)|^2|g(k_1,k_2)|^2\,\d k_1\d k_2
	<\infty\Big\}.
\end{equation}
The first condition in \eqref{cond:la.la.sqrt.in.L2} implies the 
boundedness of the perturbation $H_\alpha-H_0$ and thus the expression 
for~$H_{\alpha}$ given in \eqref{Hamiltonian} generates a self-adjoint 
operator in the Hilbert space~$\CC^2\otimes\cF_s^2$ on the natural 
domain of $H_0$ (see \cite[Theorem~V.4.3]{Kato}). Throughout the paper 
we denote the corresponding self-adjoint operator again by~$H_{\alpha}$ 
for notational convenience. The spatial dimension, $d\geq1$, plays no 
particular role in our analysis and is left arbitrary. 

Starting with the pioneering work of H\"ubner and 
Spohn~\cite{Huebner-Spohn-1995}, 
spectral properties of the spin-boson model as well as of its 
finite photon approximations have been investigated extensively and the 
corresponding literature is enormous. In contrast to rigorous results 
of the weak coupling regime, it seems that spectral properties of the 
spin-boson Hamiltonian (even with particle number cut-off) for 
general coupling have not been fully understood yet. The interested 
reader is referred, for example, to \cite{Huebner-Spohn-1995-review, 
	Huebner94atominteracting, Arai-Hirokawa-1997, 
	Derezinski-Gerard-Rev.Amth.Phys-1999, Georgescu-Gerard-Moller-CMP2004, 
	Hasler-Herbst-2011}), 
where the location of the essential spectrum is given for any value of 
the coupling constant, 
and results including the finiteness of the point spectrum or absence of 
the singular continuous spectrum 
are proven under various assumptions in addition 
to~\eqref{cond:la.la.sqrt.in.L2}. 
The ergodicity of the spin-boson Hamiltonian at arbitrary 
coupling strength was studied in~\cite{Merkli-CMP2015}. In the recent 
work \cite{Ibrogimov-AHP2018} we have obtained an explicit description 
of the essential spectrum and proved the finiteness of the discrete 
spectrum for the spin-boson model with two photons for arbitrary coupling 
where the only requirement on the coupling function was its square 
integrability. 

It is well known that, under appropriate conditions in addition to 
\eqref{cond:la.la.sqrt.in.L2}-\eqref{cond2:la.omega.in.L2}, 
there exists a sufficiently small coupling constant 
$\alpha_0>0$ such that for all $\alpha\in (0,\alpha_0)$ the self-adjoint 
operator generated by \eqref{Hamiltonian} has a unique discrete eigenvalue 
(see \cite{Minlos-Spohn-1996, Huebner-Spohn-1995}). Moreover, 
it is also well known that, if the coupling is ``weak but not very weak'', 
then one further eigenvalue appears. This is the case for any finite photon 
approximation of the spin-boson Hamiltonian 
(see~\cite{Angelescu-Minlos-Ruiz-Zagrebnov-2008, Huebner94atominteracting}). 
Unlike these observations, arbitrary coupling results of the up-to-date 
literature in this direction guarantee only finiteness of the number of 
discrete eigenvalues for a given coupling constant $\alpha>0$ 
(\cf~\cite{Ibrogimov-AHP2018}). In fact, 
it has been an open problem 
for a long time whether the spin-boson Hamiltonian 
with two photons can have more than two bound state energies 
for some (i.e.~strong) coupling (\cf~\cite[p.~192]{Minlos-Spohn-1996} and 
\cite[p.~8]{Huebner94atominteracting}). 
It is the goal of the present paper to answer this question 
for strong coupling by proving the following claim which holds under 
considerably relaxed conditions on the parameter functions.
\begin{theorem}\label{thm:main}
	Let $\la\in L^2(\RR^d)$ and let $\omega\colon\RR^d\to[0,\infty)$ 
	be an unbounded and almost everywhere continuous function 
	satisfying \eqref{zero.mass}. Then the Hamiltonian 
	$H_{\alpha}$ cannot have more than two bound state energies 
	for sufficiently strong coupling strength $\alpha>0$.
\end{theorem}
The detailed proof of this result is given in the next section. 
It relies on a simple and instructive method based on 
the splitting trick developed in author's recent 
work~\cite{Ibrogimov-AHP2018} and asymptotic 
analyses of zeros of Nevanlinna functions.

\noindent

Throughout the paper we adopt the following notation. 
For a self-adjoint operator $T$ acting in a Hilbert space and 
a constant $\mu \in \RR$ such that $\mu\leq\min\sess(T)$, we 
denote by $N(\mu; T)$ the dimension of the spectral subspace 
of $T$ corresponding to the interval $(-\infty,\mu)$. The latter 
quantity coincides with the number of discrete eigenvalues 
(counted with multiplicities) of $T$ that are less than $\mu$.
Integrals with no indication of limits imply integration 
over the whole space $\RR^d$ or $\RR^d\times\RR^d$, 
and $\left\|\cdot\right\|$ denotes the usual $L^2$-norm. 
%
%
%
%
\section{Proof of Theorem~\ref{thm:main}}
Unless otherwise specified, we always assume that the 
discrete variable $\sigma=\pm$ is fixed. Moreover, without loss 
of generality we assume that the dispersion 
relation $\omega\colon\RR^d\to[0,\infty)$ is a 
continuous function and the coupling function $\la\colon\RR^d\to\CC$ 
is not identically zero. All the arguments below plainly work 
for almost everywhere continuous $\omega$. If the coupling 
function is identically zero on~$\RR^d$, then the photons do not couple 
to the atom and the description of the spectrum becomes 
straightforward. 

It is easy to see that the transformation 
$U\colon\CC^2\otimes\cF_s^2\to\cF_s^2\oplus\cF_s^2$, defined by
\begin{equation}
	U\colon
	\begin{pmatrix}
		\begin{pmatrix}
			f_0^{(+)}\\
			f_0^{(-)}
		\end{pmatrix},
		\begin{pmatrix}
			f_1^{(+)}\\
			f_1^{(-)}
		\end{pmatrix}, 
		\begin{pmatrix}
			f_2^{(+)}\\
			f_2^{(-)}
		\end{pmatrix}
	\end{pmatrix}
	\mapsto
	\begin{pmatrix}
		\begin{pmatrix}
			f_0^{(+)}\\
			f_1^{(-)}\\
			f_2^{(+)}
		\end{pmatrix},
		\begin{pmatrix}
			f_0^{(-)}\\
			f_1^{(+)}\\
			f_2^{(-)}
		\end{pmatrix}
	\end{pmatrix},
\end{equation}
is a unitary operator and block-diagonalizes the Hamiltonian 
$H_{\alpha}$ in \eqref{Hamiltonian}, i.e. 
\begin{equation}\label{diagonalization.of.Hamiltonian}
	U^*H_{\alpha}U={\rm{diag}}\{H_{\alpha}^{(+)}, H_{\alpha}^{(-)}\},
\end{equation}
where the operator matrix
\begin{equation}\label{op.mat.H.sigma}
	H_{\alpha}^{(\sigma)}:=
	\begin{pmatrix}
		H^{(\sigma)}_{00} & \alpha H_{01} & 0\\[0.5ex]
		\alpha H_{10} & H^{(\sigma)}_{11} & \alpha H_{12}\\[0.5ex]
		0 & \alpha H_{21} & H^{(\sigma)}_{22}
	\end{pmatrix}, \quad \Dom(H_{\alpha}^{(\sigma)})
	:=\CC\oplus\cH_1\oplus\cH_2,
\end{equation}
acts in the truncated Fock space~$\cF_s^2$ with~$\cH_1$ and~$\cH_2$ 
defined in~\eqref{Dom:H1} and~\eqref{Dom:H2}. The operator 
entries of $H_{\alpha}^{(\sigma)}$ are given by 
\begin{equation}
	\begin{aligned}
		&H^{(\sigma)}_{00} f_0 = \sigma\eps f_0, \quad  H_{01}f_1 
		= \int \la(q)f_1(q) \,\d q, \quad (H_{10}f_0)(k)
		=f_0\ov{\la(k)},\\ \label{def.op.entries.2}
		&(H^{(\sigma)}_{11} f_1)(k)=(-\sigma\eps+\omega(k))f_1(k), 
		\quad (H_{12}f_2)(k)= \int f_2(k,q)\la(q) \,\d q,\\ 
		&(H_{21}f)(k_1,k_2) = \ov{\la(k_1)}f(k_2)+\ov{\la(k_2)}f(k_1)
	\end{aligned}
\end{equation}
and 
\begin{equation*}
	(H^{(\sigma)}_{22}f_2)(k_1,k_2)=(\sigma\eps+\omega(k_1)+\omega(k_2))f_2(k_1,k_2), 
	\quad (f_0,f_1,f_2)\in\cF_s^2.
\end{equation*}
It follows from the square-integrability of the 
coupling function that 
$H_{12}\colon L^2_s(\RR^d\times\RR^d)\to L^2(\RR^d)$ 
and $H_{01}\colon L^2(\RR^d)\to\CC$ are bounded operators 
with $H_{10}=H_{01}^*$ and $H_{21}=H_{12}^*$. 
Hence, $H_{\alpha}^{(\sigma)}$ is a self-adjoint operator on 
the domain $\CC\oplus\cH_1\oplus\cH_2$ (see \cite[Theorem~V.4.3]{Kato}). 

In view of \eqref{diagonalization.of.Hamiltonian}, 
in the sequel we focus on the study of the discrete spectrum 
of the operator matrices~$H_{\alpha}^{(\pm)}$. We recall 
from~\cite{Ibrogimov-AHP2018} that
\begin{equation}\label{bottom.ess.spec.op.mat}
	\begin{aligned}
		\min\sess(\wh H_{\alpha}^{(+)})&=E_{\eps}(\alpha) 
		\quad \text{for all} \quad \alpha>0,\\[1ex]
		\min\sess(\wh H_{\alpha}^{(-)})&=
		\begin{cases} 
			-\eps & \mbox{if } \quad\ds\frac{\la}{\sqrt{\omega}}\in L^2(\RR^d) 
			\mbox{ and } \ds 0<\alpha\leq\frac{\sqrt{2\eps}}
			{\bigl\|\frac{\la}{\sqrt{\omega}}\bigr\|}, \\ 
			E_{-\eps}(\alpha) & \mbox{otherwise}, 
		\end{cases}
	\end{aligned}
\end{equation}
where $E_{\sigma\eps}(\alpha)$ is the unique zero of the continuous 
function
\begin{equation}\label{Phi}
	\Phi^{(\sigma)}_{\alpha}(z)=-\sigma\eps-z-
	\alpha^2\int\frac{|\la(q)|^2 \,\d q}{\omega(q)+\sigma\eps-z}, 
	\quad z\in(-\infty,\sigma\eps).
\end{equation}
We notice that $E_\eps(\alpha)$ exists for all $\alpha>0$, while the 
existence of $E_{-\eps}(\alpha)$ requires $\alpha>0$ to be not too 
small whenever the condition $\la/\sqrt{\omega}\in L^2(\RR^d)$ holds. 
We make the convention~$E_{-\eps}(\alpha):=-\eps$ 
whenever~$\Phi^{(-)}_\alpha$ does not have a zero in $(-\infty,-\eps)$.
In view of~\eqref{diagonalization.of.Hamiltonian}, the bottom 
of the essential spectrum of the Hamiltonian~$H_\alpha$ is thus 
given by 
\begin{equation}\label{ess.spec.Hamiltonian}
	E(\alpha):=\min\sess(H^{(\sigma)}_\alpha):=
	\min\{E_{\eps}(\alpha), E_{-\eps}(\alpha)\}.
\end{equation}
One can easily observe that $E_\eps(\alpha)<-\eps$ for 
all~$\alpha>0$. 

\smallskip

In the sequel we will be dealing with the non-linear pencil 
$R_{\alpha}^{(\sigma)}\colon(-\infty, \sigma\eps)\to\cL(\cH_1, L^2(\RR^d))$, 
defined by
\begin{equation}
	R_{\alpha}^{(\sigma)}(z):=H^{(\sigma)}_{11}-z
	-\alpha^2H_{12}(H^{(\sigma)}_{22}-z)^{-1}H_{12}^*
	-\alpha^2H_{01}(H^{(\sigma)}_{00}-z)^{-1}H_{01}^*.
\end{equation}
It is easy to see that, for each $z<\sigma\eps$, the operator 
$R_{\alpha}^{(\sigma)}(z)$ is well-defined, bounded from below 
and self-adjoint on the Hilbert space $\cH_1$. 

\smallskip

We split the proof of Theorem~\ref{thm:main} into several steps as in the 
following lemmas. 
\begin{lemma}\label{lem:crucial}
	The number of the discrete eigenvalues of the operator matrix 
	$H_{\alpha}^{(\sigma)}$ is equal to the number of the negative 
	eigenvalues of the operator $R_{\alpha}^{(\sigma)}(E_{\sigma\eps}(\alpha))$, i.e.
	\begin{equation}\label{disc.spec.charac.Schur}
		N\bigl(E_{\sigma\eps}(\alpha); H_{\alpha}^{(\sigma)}\bigr) 
		= N\bigl(0; R_{\alpha}^{(\sigma)}(E_{\sigma\eps}(\alpha))\bigr).
	\end{equation}	
\end{lemma}
\begin{proof}
	First, we justify the relation 	
	\begin{equation}\label{disc.spec.mat.Schur}
		z\in\sd(H^{(\sigma)}_\alpha)\cap(-\infty, E_{\sigma\eps}(\alpha)) 
		\quad 
		\Longleftrightarrow \quad 0\in\sd(R_{\alpha}^{(\sigma)}(z)).
	\end{equation}
	To this end, let us fix $z<E_{\sigma\eps}(\alpha)$ and note that $z$ is a 
	discrete eigenvalue for the operator matrix $H^{(\sigma)}_\alpha$ with a 
	non-zero eigenvector $\Psi=(f_0, f_1, f_2)^t\in\cF^2_s$ if and only if 
	%
	\begin{equation}\label{system2}
		\begin{cases}
			\begin{aligned}
				(H^{(\sigma)}_{00}-z)f_0+\alpha H_{01}f_1=0,\\[1.5ex]
				\alpha H_{10}f_0+(H^{(\sigma)}_{11}-z)f_1+\alpha H_{12}f_2=0, \\[1.5ex]
				\alpha H_{21}f_1+(H^{(\sigma)}_{22}-z)f_2=0.
			\end{aligned}
		\end{cases}
	\end{equation}
	Since the spectra of both $H^{(\sigma)}_{00}$ and $H^{(\sigma)}_{22}$ 
	lie on the right of $\sigma\eps$, the first and the third equations in~\eqref{system2} 
	can be equivalently written as 
	$f_0=-\alpha(H^{(\sigma)}_{00}-z)^{-1}H_{01}f_1$ and 
	$f_2=-\alpha(H^{(\sigma)}_{22}-z)^{-1}H_{21}f_1$, respectively. 
	Inserting these into the second equation in \eqref{system2}, we obtain 
	\begin{equation}\label{disc.spec.Schur.comp}
		R_{\alpha}^{(\sigma)}(z)f_1=0. 
	\end{equation}
	Note that $f_1$ must be a non-zero vector, for otherwise $\Psi$ would be 
	the zero element of $\cF_s^2$, contradicting our hypothesis. 
	Since $z<E_{\sigma\eps}(\alpha)$, it follows that 
	$R_{\alpha}^{(\sigma)}(z)$ is a Fredholm operator as it is a rank-one 
	perturbation of the Fredholm operator corresponding to the Schur 
	complement of $H^{(\sigma)}_{22}-z$ in the lower $2\times2$ operator 
	matrix in $H^{(\sigma)}_\alpha$ (see \cite{Ibrogimov-AHP2018}). 
	That is why \eqref{disc.spec.Schur.comp} is equivalent to the fact that 
	$0$ is a discrete eigenvalue for $R_{\alpha}^{(\sigma)}(z)$ with an 
	eigenvector $f_1$, thus justifying the claim in \eqref{disc.spec.mat.Schur}. 

	\smallskip
	
	Now the claim in \eqref{disc.spec.charac.Schur} follows from \eqref{disc.spec.mat.Schur} 
	if we show that $R_{\alpha}^{(\sigma)}$ 
	is a strictly decreasing operator function in $(-\infty, E_{\sigma\eps}(\alpha))$ 
	and has the following property
	\begin{equation}\label{lim.inf.Rz}
		R_{\alpha}^{(\sigma)}(z)\uparrow+\infty \quad \text{as} \quad z\downarrow-\infty. 
	\end{equation}
	To this end, let us fix $\phi\in\cH_1$ and note that
	\begin{equation}\label{quad.form.R}
		\begin{aligned}
			\bigl\langle R_{\alpha}^{(\sigma)}(z)\phi, \phi\bigr\rangle
			=&\langle H^{(\sigma)}_{11}\phi,\phi\rangle-z\|\phi\|^2
			-\alpha^2\bigl\langle (H^{(\sigma)}_{22}-z)^{-1}H_{12}^*\phi, 
			H_{12}^*\phi\bigr\rangle\\
			&-\alpha^2\bigl\langle (H^{(\sigma)}_{00}-z)^{-1}H_{01}^*\phi, 
			H_{01}^*\phi\bigr\rangle.
		\end{aligned}
	\end{equation}
	Hence, it follows that, for $\phi\neq0$,
	\begin{equation}\label{deriv.wrt.z}
		\frac{\partial}{\partial z}\bigl\langle R_{\alpha}^{(\sigma)}(z)\phi, 
		\phi\bigr\rangle
		=-\|\phi\|^2-\alpha^2\|(H^{(\sigma)}_{22}-z)^{-1}H_{12}^*\phi\|^2
		-\alpha^2\|(H^{(\sigma)}_{00}-z)^{-1}H_{01}^*\phi\|^2<0.
	\end{equation}
	Moreover, we have the following standard estimates 
	\begin{equation}
		\begin{aligned}
			\bigl|\bigl\langle R_{\alpha}^{(\sigma)}(z)\phi, \phi\bigr\rangle
			-\langle H^{(\sigma)}_{11}\phi,\phi\rangle+z\|\phi\|^2\bigr|
			& \leq \bigl|\bigl\langle (H^{(\sigma)}_{22}-z)^{-1}H_{12}^*\phi, 
			H_{12}^*\phi\bigr\rangle\bigr|
			+\bigl|\bigl\langle (H^{(\sigma)}_{00}-z)^{-1}H_{01}^*\phi, 
			H_{01}^*\phi\bigr\rangle\bigr|\\
			&\leq\|(H^{(\sigma)}_{22}-z)^{-1}\|\|H_{12}^*\phi\|^2
			+
			\|(H^{(\sigma)}_{00}-z)^{-1}\|\|H_{01}^*\phi\|^2\\
			&\leq \frac{1}{\sigma\eps-z}\bigl(\|H_{12}^*\phi\|^2
			+\|H_{01}^*\phi\|^2\bigr).				
		\end{aligned}	
	\end{equation}
	Now letting $z\downarrow-\infty$, we obtain the claim in \eqref{lim.inf.Rz}.	
\end{proof}
We notice that the function $\Phi^{(-)}_{\alpha}$ (see \eqref{Phi}) 
has a unique zero for all sufficiently large $\alpha>0$. 
On the account of Lemma~\ref{lem:crucial}, from now on we always 
assume~$\alpha>0$ to be so large that the bottom of the essential 
spectrum of $H^{(\sigma)}_\alpha$ coincides with the unique zero in 
$(-\infty,\sigma\eps)$ of $\Phi^{(\sigma)}_{\alpha}$.

\begin{lemma}\label{asymp.of.E.alpha}
	The bottom of the essential spectrum of $H^{(\sigma)}_\alpha$ has the following 
	large coupling behavior 
	\begin{equation}\label{asymp.exp.Ea}
		E_{\sigma\eps}(\alpha)=-\|\la\|\alpha+\rm{o}(\alpha), 
		\quad \alpha\uparrow+\infty.
	\end{equation}
\end{lemma}
\begin{proof}
	Consider the continuous function
	\begin{equation}
		\psi(x,y)=y-\sigma\eps x-\frac{1}{\|\la\|^2}\int 
		\frac{|\la(q)|^2 \,\d q}{(\omega(q)+\sigma\eps)x+y}
	\end{equation}
	for $(x,y)\in[0,\frac{1}{\eps})\times[\frac{1}{2},\infty)$. 
	We have $\psi(0,1)=0$ and the partial derivative 
	$\frac{\partial\psi}{\partial y}$ exists and is right continuous 
	at $(0,1)$. Since $\frac{\partial\psi}{\partial y}(0,1)=2
	$, 
	the implicit function theorem\footnote{In fact, here we are using 
		the non-standard version (\cf~\cite{Kumagai-1980-implicit}) 
		of the implicit function theorem which requires only strict 
		monotonicity in $y$ and no differentiability in a (right-) 
		neighborhood of $(0,1)$.} 
	applies and yields the existence of a sufficiently small 
	constant $\delta>0$ and a unique continuous function 
	$\wh E_{\sigma\eps}\colon[0,\delta)\to\RR$ such that $\wh E(0)=1$ 
	and $\psi(\beta, \wh E_{\sigma\eps}(\beta))=0$ for all 
	$\beta\in[0,\delta)$. Therefore, we get 
	\begin{equation}\label{asymp.1}
		\Phi^{(\sigma)}_{\alpha}\bigg(-\alpha\|\la\|
		\wh E_{\sigma\eps}\Bigl(\frac{1}{\alpha\|\la\|}\Bigr)\bigg)
		=\alpha\|\la\|\psi\bigg(\frac{1}{\alpha\|\la\|}, 
		\wh E_{\sigma\eps}\Bigl(\frac{1}{\alpha\|\la\|}
		\Bigr)\bigg)=0 
	\end{equation}
	for all $\alpha\in\bigl(\frac{1}{\delta\|\la\|}, \infty)$. 
	Clearly, we have 
	\begin{equation}
		-\alpha\|\la\|\wh E_{\sigma\eps}\Bigl(\frac{1}{\alpha\|\la\|}
		\Bigr)<\sigma\eps
	\end{equation}
	for all sufficiently large $\alpha>0$. Since  
	$E_{\sigma\eps}(\alpha)$ is the unique zero of the function 
	$\Phi^{(\sigma)}_{\alpha}$ in the interval $(-\infty,\sigma\eps)$, 
	we thus deduce from \eqref{asymp.1} that, for all sufficiently large $\alpha>0$,
	\begin{equation}
		E_{\sigma\eps}(\alpha)=-\alpha\|\la\|\wh E_{\sigma\eps}
		\Bigl(\frac{1}{\alpha\|\la\|}\Bigr).
	\end{equation}
	This and the 
	right-continuity at zero of the function $\wh E_{\sigma\eps}(\mydot)$ yield 
	the desired asymptotic expansion.
\end{proof}
It is easy to verify that  
\begin{equation}\label{Schur2}
	R_{\alpha}^{(\sigma)}(z) = \Delta_{\alpha}^{(\sigma)}(z)-\alpha^2 \wh K^{(\sigma)}(z),
\end{equation}
where, $\Delta_{\alpha}^{(\sigma)}\colon(-\infty, \sigma\eps)\to\cL(\cH_1, L^2(\RR^d))$ is 
the pencil of multiplication operators by the functions 
\begin{equation}\label{def:Delta}
	\begin{aligned}
		\Delta_{\alpha}^{(\sigma)}(k;z):
		&= \Phi_{\alpha}^{(\sigma)}(z-\omega(k))\\
		&=\omega(k)-\sigma\eps-z-\alpha^2\int\frac{|\la(q)|^2\, \d q}{\omega(k)+\omega(q)+\sigma\eps-z}
	\end{aligned}
\end{equation}
and $\wh K^{(\sigma)}(z)\colon(-\infty, \sigma\eps)\to\cL(\cH_1, L^2(\RR^d))$ is the 
pencil of integral operators with the kernels 
\begin{equation}\label{kernel:p}
	\wh p^{(\sigma)}(k_1,k_2;z):=
	\frac{\ov{\la(k_1)}\la(k_2)}{\omega(k_1)+\omega(k_2)+\sigma\eps-z} 
	+\frac{\ov{\la(k_1)}\la(k_2)}{\sigma\eps-z}.
\end{equation}
For fixed $z\leq E_{\sigma\eps}(\alpha)$, let us consider the decomposition
\begin{equation}
	\frac{1}{\omega(k_1)+\omega(k_2)+\sigma\eps-z}=\Psi^{(\sigma)}_1(k_1,k_2;z)
	+\Psi^{(\sigma)}_2(k_1,k_2;z),
\end{equation}
where
\begin{equation}\label{kernel:k_1} 						 	
	\Psi^{(\sigma)}_1(k_1,k_2;z):=\frac{1}{\omega(k_1)+\sigma\eps-z}+
	\frac{1}{\omega(k_2)+\sigma\eps-z}-\frac{1}{\sigma\eps-z}
\end{equation}
and
\begin{equation}\label{kernel:k_2}
	\Psi^{(\sigma)}_2(k_1,k_2;z):=\frac{1}{\omega(k_1)+\omega(k_2)+\sigma\eps-z}
	-\Psi^{(\sigma)}_1(k_1,k_2;z),\quad k_1,k_2\in\RR^d.
\end{equation}
We denote by $\wh K^{(\sigma)}_1(z)$ and $K^{(\sigma)}_2(z)$ the integral operators 
in $L^2(\RR^d)$ whose kernels are the functions 
\begin{equation}
	(k_1,k_2)\mapsto\ov{\la(k_1)}\la(k_2)\Bigl[\Psi^{(\sigma)}_1(k_1,k_2;z)+\frac{1}{\sigma\eps-z}\Bigr]
\end{equation}
and $(k_1,k_2)\mapsto\ov{\la(k_1)}\la(k_2)\Psi^{(\sigma)}_2(k_1,k_2;z)$, respectively. 
Then, we have the corresponding decomposition
\begin{equation}\label{decomposition.of.K}
	\wh K^{(\sigma)}(z)=\wh K^{(\sigma)}_1(z)+K^{(\sigma)}_2(z).
\end{equation} 
The next result confirms that the numerical range of the self-adjoint operator $\Delta_{\alpha}^{(\sigma)}(E_{\sigma\eps}(\alpha))-\alpha^2K^{(\sigma)}_2(E_{\sigma\eps}(\alpha))$ 
is contained in $(0,\infty)$ for all sufficiently large $\alpha>0$. 
\begin{lemma}\label{lem:positive.num.range}
	Let $\psi\in\cH_1$ be arbitrary. For all sufficiently large $\alpha>0$, we have
	\begin{equation}\label{num.range.estimate}
		\int\Delta^{(\sigma)}_\alpha(k;E_{\sigma\eps}(\alpha))|\psi(k)|^2\,{\rm{d}} k \geq \alpha^2 \int \bigl(K^{(\sigma)}_2(E_{\sigma\eps}(\alpha))\psi\bigr)(k)\ov{\psi(k)}\,{\rm{d}} k.
	\end{equation}	
\end{lemma}
\begin{proof}
	Since $\Phi^{(\sigma)}_{\alpha}(E_{\sigma\eps}(\alpha))=0$, we have 
	the following simple yet quite important pointwise estimate	
	\begin{equation*}
		\begin{aligned}
			\Delta_{\alpha}^{(\sigma)}(k;E_{\sigma\eps}(\alpha))&= 
			\Delta_{\alpha}^{(\sigma)}(k;E_{\sigma\eps}(\alpha))
			-\Phi^{(\sigma)}_{\alpha}(E_{\sigma\eps}(\alpha))\\
			&=\omega(k)\bigg(1+\alpha^2\int\frac{|\la(q)|^2\,\d q}
			{(\omega(q)+\sigma\eps-E_{\sigma\eps}(\alpha))(\omega(q)+\omega(k)
				+\sigma\eps-E_{\sigma\eps}(\alpha))}\bigg)\\
			&\geq \omega(k), \quad k\in\RR^d.
		\end{aligned}
	\end{equation*}
	This leads to the following estimate for the left-hand-side of 
	\eqref{num.range.estimate}  
	\begin{equation}\label{num.range.low.bd}
		\int\Delta^{(\sigma)}_\alpha(k;E_{\sigma\eps}(\alpha))|\psi(k)|^2\,\d k \geq 
		\int\omega(k)|\psi(k)|^2\,\d k.
	\end{equation}
	Next, we recall from \cite{Ibrogimov-AHP2018} the following elementary 
	(yet quite important!) inequality
	\begin{equation}\label{elementary.ineq}
		0 \leq \frac{1}{a+b+c}-\frac{1}{a+c}-\frac{1}{b+c}
		+\frac{1}{c} \leq  \frac{\sqrt{ab}}{2c^2},
	\end{equation}
	which holds for all $a\geq0$, $b\geq0$ and $c>0$. Applying this 
	inequality with $a=\omega(k_1)\geq0$, $b=\omega(k_2)\geq0$ 
	and $c=\sigma\eps-E_{\sigma\eps}(\alpha)>0$, we can estimate the 
	function in \eqref{kernel:k_2} as follows
	\begin{equation}\label{estimate.pointwise.on.Psi2}
		0\leq \Psi^{(\sigma)}_2(k_1,k_2;E_{\sigma\eps}(\alpha)) \leq 
		\frac{1}{2(\sigma\eps-E_{\sigma\eps}(\alpha))^2}\sqrt{\omega(k_1)}\sqrt{\omega(k_2)}, 
		\quad k_1,k_2\in\RR^d.
	\end{equation}
	Using Fubini's theorem and the Cauchy-Schwarz inequality we thus obtain the following 
	upper bound for the right-hand-side of \eqref{num.range.estimate}
	\begin{equation}\label{CS}
		\begin{aligned}
			\bigg|\alpha^2 \int \bigl(K^{(\sigma)}_2(E_{\sigma\eps}(\alpha))\psi\bigr)(k)
			\ov{\psi(k)}\,\d k\bigg|
			&\leq \frac{\alpha^2}{2(\sigma\eps-E_{\sigma\eps}(\alpha))^2}
			\bigg(\int|\la(k)||\psi(k)|\sqrt{\omega(k)}\,\d k\bigg)^2\\
			&\leq \frac{\alpha^2 \|\la\|^2}{2(\sigma\eps-E_{\sigma\eps}(\alpha))^2}
				\bigg(\int\omega(k)|\psi(k)|^2\,\d k\bigg).
		\end{aligned}
	\end{equation}
	On the other hand, the asymptotic expansion \eqref{asymp.exp.Ea} guarantees that
	\begin{equation}\label{asymp.const.integ.estim}
		\frac{\alpha^2}{2(\sigma\eps-E_{\sigma\eps}(\alpha))^2}
		=\frac{\alpha^2}{2(\sigma\eps+\|\la\|\alpha+\rm{o}(\alpha))^2}
		=\frac{1}{2\|\la\|^2}+\rm{o}(1), \quad \alpha\uparrow +\infty.
	\end{equation}
	In view of this, \eqref{num.range.estimate} follows 
	immediately from \eqref{CS} combined with \eqref{num.range.low.bd}.
\end{proof}
\begin{lemma}\label{lem:K1.has.1.EV}
	For all sufficiently large $\alpha>0$, the integral operator 
	$\wh K^{(\sigma)}_1(E_{\sigma\eps}(\alpha))$ has exactly one 
	positive eigenvalue.	
\end{lemma}
\begin{proof}
	First, we recall that the kernel of the integral operator 
	$\wh K^{(\sigma)}_1(E_{\sigma\eps}(\alpha))$ is given by the function
	\begin{equation}
		(k_1,k_2)\mapsto\ov{\la(k_1)}\la(k_2)\Bigl[\frac{1}{\omega(k_1)
			+\sigma\eps-E_{\sigma\eps}(\alpha)}
		+\frac{1}{\omega(k_2)+\sigma\eps-E_{\sigma\eps}(\alpha)}\Bigr].
	\end{equation}
	Observe that $\wh K^{(\sigma)}_1(E_{\sigma\eps}(\alpha))$ is a well-defined 
	rank-two operator in $L^2(\RR^d)$. In fact, its range coincides with the span 
	of the linearly independent functions $\ov{\la}$ and 
	$\frac{\ov{\la}}{\omega+\sigma\eps-E_{\sigma\eps}(\alpha)}$. It is easy to 
	check that the matrix $M_\alpha^{(\sigma)}=(m_{ij})^2_{i,j=1}$ of 
	$\wh K^{(\sigma)}_1(E_{\sigma\eps}(\alpha))$ with respect to the basis 
	consisting of the latter two functions has the following entries 
	\begin{equation}
		\begin{aligned}
			&m_{11}=\int_{\RR^d}\frac{|\la(q)|^2\,\d q}{\omega(q)
				+\sigma\eps-E_{\sigma\eps}(\alpha)}, \quad
			m_{12}=\int_{\RR^d}\frac{|\la(q)|^2\,\d q}{(\omega(q)+\sigma\eps
				-E_{\sigma\eps}(\alpha))^2},\\
			&m_{21}=\int_{\RR^d}|\la(q)|^2\, \d q, \quad m_{22}=m_{11}.
		\end{aligned}
	\end{equation}
	By the Cauchy-Schwarz inequality, we have
	\begin{equation}
		\det M_\alpha^{(\sigma)}=\bigg(\int_{\RR^d}\frac{|\la(q)|^2\,\d q}{\omega(q)+\sigma\eps-E_{\sigma\eps}(\alpha)}\bigg)^2-\int_{\RR^d}|\la(q)|^2\, \d q\int_{\RR^d}\frac{|\la(q)|^2\,\d q}{(\omega(q)+\sigma\eps-E_{\sigma\eps}(\alpha))^2}
		\leq 0.	
	\end{equation}
	However, the equality sign cannot hold as the functions $\ov{\la}$ 
	and $\frac{\ov{\la}}{\omega+\sigma\eps-E_{\sigma\eps}(\alpha)}$ are linearly independent. 
	Hence, precisely one of the two eigenvalues of the matrix 
	$M_\alpha^{(\sigma)}\in\RR^{2\times2}$ (and thus of the 
	integral operator $\wh K^{(\sigma)}_1(E_{\sigma\eps}(\alpha))$) is positive. 
\end{proof}
\paragraph{Proof of Theorem~\ref{thm:main}.}
In view of Lemma~\ref{lem:positive.num.range}, we have 
\begin{equation}
	\begin{aligned}
		R_{\alpha}^{(\sigma)}(E_{\sigma\eps}(\alpha))
		&=\Delta_{\alpha}^{(\sigma)}(E_{\sigma\eps}(\alpha))
		-\alpha^2K^{(\sigma)}_1(E_{\sigma\eps}(\alpha))
		-\alpha^2K^{(\sigma)}_1(E_{\sigma\eps}(\alpha))\\
		&\geq -\alpha^2K^{(\sigma)}_1(E_{\sigma\eps}(\alpha))
	\end{aligned}
\end{equation}
for all sufficiently large $\alpha>0$. Hence, the variational 
principle (see e.g.~\cite{Reed-Simon-IV})
and Lemma~\ref{lem:K1.has.1.EV} 
imply that
\begin{equation}\label{Schur.has.atm1.EV}
	N\bigl(0; R_{\alpha}^{(\sigma)}(E_{\sigma\eps}(\alpha))\bigr)\\
	\leq  N\bigl(0; -\alpha^2\wh K^{(\sigma)}_1(E_{\sigma\eps}(\alpha))\bigr)=1.
\end{equation}
Combining \eqref{diagonalization.of.Hamiltonian}, Lemma~\ref{lem:crucial} and 
\eqref{Schur.has.atm1.EV}, we conclude that 
\begin{equation}
	\begin{aligned}
		N\bigl(E(\alpha); H_{\alpha}\bigr) &\leq 
			\sum_{\sigma=\pm} N\bigl(E_{\sigma\eps}(\alpha); H_{\alpha}^{(\sigma)}\bigr)\\
		&= \sum_{\sigma=\pm} N\bigl(0; R_{\alpha}^{(\sigma)}(E_{\sigma\eps}(\alpha))\bigr)\\
		&\leq2.
	\end{aligned}
\end{equation}
This completes the proof of Theorem~\ref{thm:main}.
\qed
\subsection*{Acknowledgments}
The author is grateful to Prof.~Alex Sobolev and Prof.~Pavel Exner 
for stimulating discussions and thanks Dr.~Jean-Claude Cuenin 
for reading the present version of the manuscript and making 
constructive remarks.
%
%
%
{\small
	\bibliographystyle{acm}
	\bibliography{spin_boson_literature_20190202}
}
\end{document}